\documentclass{amsart}

\usepackage{amsfonts}
\usepackage{amsmath}
\usepackage{amssymb}
\usepackage{hyperref}

\newtheorem{theorem}{Theorem}
\theoremstyle{plain}

\newtheorem{corollary}{Corollary}

\newtheorem{lemma}{Lemma}

\newtheorem{remark}{Remark}

\numberwithin{equation}{section}

\begin{document}

\title[Some Steffensen's type inequalities]{Some Steffensen's type inequalities}

\author[M.W. Alomari]{M.W. Alomari$^{1,*}$}
\address{$^{1}$Department of Mathematics, Faculty of Science and
Information Technology, Irbid National University, 2600 Irbid
21110, Jordan.} \email{mwomath@gmail.com}

\author[S. Hussain]{S. Hussain$^2$}
\address{$^2$Department of Mathematics, University of Engineering and  Technology, Lahore, Pakistan}
\email{sabirhus@gmail.com }

\author[Z. Liu]{Z. Liu$^{3}$}
\address{$^{3}$Institute of Applied Mathematics, School of Science University of
Science and Technology, Liaoning Anshan 114051, Liaoning, China}.
\email{lewzheng@163.net}

\thanks{$^*$Corresponding author}

\date{December 30, 2012.}


\begin{abstract}
In this paper, new inequalities connected with the celebrated
Steffensen's integral inequality are proved.
\end{abstract}

\maketitle

\section{Introduction}
In 1918 and in order to study certain inequalities between mean
values, Steffensen \cite{RefN} has proved the following inequality
(see also \cite{RefJ} \& \cite{RefK}):
\begin{theorem}
Let $f$ and $g$ be two integrable functions defined on $(a,b)$,
$f$ is decreasing and for each $t \in (a,b)$, $0 \le g(t) \le 1$.
Then, the following inequality
\begin{eqnarray}
\label{Stef}\int\limits_{b - \lambda }^b {f\left( t \right)dt}
\le \int\limits_a^b {f\left( t \right)g\left( t \right)dt} \le
\int\limits_a^{a + \lambda } {f\left( t \right)dt}
\end{eqnarray}
holds, where, $\lambda  = \int\limits_a^b {g\left( t \right)dt}$.
\end{theorem}

Some minor generalization of Steffensen's inequality (\ref{Stef})
was considered by T. Hayashi \cite{RefD}, using the substituting $
{{g\left( t \right)} \mathord{\left/
 {\vphantom {{g\left( t \right)} A}} \right.
 \kern-\nulldelimiterspace} A}$ for $g\left( t \right)$, where $A$ is positive constant.
In 1953, another interesting result was proved by Ap\'{e}ry
\cite{RefA}, where he extended the Steffensen's inequality
(\ref{Stef}) on infinite interval. Namely, we have the following
result:
\begin{theorem}
Let $f$ be decreasing on $(0, \infty)$ and let $g$ be a measurable
function on $(0, \infty)$ such that $0 \le g(t) \le A$ ($A\neq
0$). Then
\begin{eqnarray}
\label{Apery} \int\limits_{0}^{\infty} {f\left( t \right)g\left( t
\right)dt} \le A \int\limits_0^{\lambda } {f\left( t \right)dt},
\end{eqnarray}
where, $\lambda  = \frac{1}{A}\int\limits_0^{\infty} {g\left( t
\right)dt}$.
\end{theorem}
For more results concerning new proofs, generalizations, weaker
hypothesis or different forms were emerging one after another see
\cite{RefA}--\cite{RefN}, and the references therein.
\newline

The aim of this paper is to establish some Steffensen's type
inequalities under various assumptions.

\section{The results }

We may start with the following lemma:
\begin{lemma}\cite{RefK}
\label{lemma1}Let $f,g: [a,b] \to \mathbb{R}$ be integrable such
that $0 \le g(t) \le 1$, for all $t \in [a,b]$ and $\int_a^b
{g\left( t \right)df\left( t \right)}$  exists. Then we have the
following representation
\begin{multline}
\label{eq2.1}\int_a^{a + \lambda } {f\left( t \right)dt}  -
\int_a^b {f\left( t \right)g\left( t \right)dt}
\\
= -  \int_a^{a + \lambda } {\left( {\int_a^x {\left( {1 - g\left(
t \right)} \right)dt} } \right)df\left( x \right)} - \int_{a +
\lambda }^b {\left( {\int_x^b {g\left( t \right)dt} }
\right)df\left( x \right)} ,
\end{multline}
and
\begin{multline}
\label{eq2.2} \int_a^b {f\left( t \right)g\left( t \right)dt -
\int_{b - \lambda }^b {f\left( t \right)dt} }
\\
=  - \int_a^{b - \lambda } {\left( {\int_a^x {g(t)dt} }
\right)df\left( x \right)}  - \int_{b - \lambda }^b {\left(
{\int_x^b {\left( {1 - g\left( x \right)} \right)dt} }
\right)df\left( x \right)},
\end{multline}
where $\lambda : = \int_a^b {g\left( t \right)dt}$.
\end{lemma}

\begin{proof}
Integrating by parts
\begin{align*}
&-  \int_a^{a + \lambda } {\left( {\int_a^x {\left( {1 - g\left( t
\right)} \right)dt} } \right)df\left( x \right)} - \int_{a +
\lambda }^b {\left( {\int_x^b {g\left( t \right)dt} }
\right)df\left( x \right)}
\\
&= - \left( {\int_a^{a + \lambda } {\left( {1 - g\left( t \right)}
\right)dt} } \right)f\left( {a + \lambda } \right) + \int_a^{a +
\lambda } {f\left( x \right)d\left( {\int_a^x {\left( {1 - g\left(
t \right)} \right)dt} } \right)}
\\
&\qquad + \left( {\int_{a + \lambda }^b {g\left( t \right)dt} }
\right)f\left( {a + \lambda } \right) + \int_{a + \lambda }^b
{f\left( x \right)d\left( {\int_x^b {g\left( t \right)dt} }
\right)}
\\
&=  - \left( {\int_a^{a + \lambda } {\left( {1 - g\left( t
\right)} \right)dt} } \right)f\left( {a + \lambda } \right) +
\int_a^{a + \lambda } {f\left( x \right)\left( {1 - g\left( x
\right)} \right)dx}
\\
&\qquad + \left( {\int_{a + \lambda }^b {g\left( t \right)dt} }
\right)f\left( {a + \lambda } \right) - \int_{a + \lambda }^b
{f\left( x \right)g\left( x \right)dx}
\\
&=  - \lambda f\left( {a + \lambda } \right) + f\left( {a +
\lambda } \right)\int_a^{a + \lambda } {g\left( t \right)dt}  +
\int_a^{a + \lambda } {f\left( x \right)dx}
\\
&\qquad - \int_a^{a + \lambda } {f\left( x \right)g\left( x
\right)dx} + f\left( {a + \lambda } \right)\int_{a + \lambda }^b
{g\left( t \right)dt}  - \int_{a + \lambda }^b {f\left( x
\right)g\left( x \right)dx}
\\
&=- \lambda f\left( {a + \lambda } \right) + f\left( {a + \lambda
} \right)\int_a^{a + \lambda } {g\left( t \right)dt}  + f\left( {a
+ \lambda } \right)\int_{a + \lambda }^b {g\left( t \right)dt}
\\
&\qquad + \int_a^{a + \lambda } {f\left( x \right)dx}  - \int_a^b
{f\left( x \right)g\left( x \right)dx}
\\
&=\int_a^{a + \lambda } {f\left( x \right)dx}  - \int_a^b {f\left(
x \right)g\left( x \right)dx},
\end{align*}
which gives the desired representation (\ref{eq2.1}). The identity
(\ref{eq2.2}) can be also proved in a similar way, we shall omit
the details.
\end{proof}

\subsection{Inequalities for bounded variation integrators}

Our first result may be stated as follows:
\begin{theorem}
\label{thm.BV}Let $f,g: [a,b] \to \mathbb{R}$ be such that $0 \le
g(t) \le 1$, for all $t \in [a,b]$ and $\int_a^b {g\left( t
\right)df\left( t \right)}$  exists. If $f$ is of bounded
variation on $[a,b]$, then we have
\begin{align}
\label{eq2.3}\left| {\int_a^{a + \lambda } {f\left( t \right)dt} -
\int_a^b {f\left( t \right)g\left( t \right)dt} } \right| &\le
\left[ {\int_{a + \lambda }^b {g\left( t \right)dt} } \right]
\cdot \bigvee_a^b \left( f \right)
\\
&\le \left( {b - a - \lambda } \right) \cdot \bigvee_a^b \left( f
\right),\nonumber
\end{align}
and
\begin{align}
\label{eq2.4}\left| {\int_a^b {f\left( t \right)g\left( t
\right)dt}  - \int_{b - \lambda }^b {f\left( t \right)dt} }
\right| &\le \left[ {\int_a^{b - \lambda } {g\left( t \right)dt} }
\right] \cdot \bigvee_a^b \left( f \right)
\\
&\le \left( {b - a - \lambda } \right) \cdot \bigvee_a^b \left( f
\right),\nonumber
\end{align}
where $\lambda : = \int_a^b {g\left( t \right)dt}$.
\end{theorem}

\begin{proof}
We prove the inequality (\ref{eq2.3}). Taking the modulus in
(\ref{eq2.1}) and utilizing the triangle inequality, we get
\begin{multline*}
\left| {\int_a^{a + \lambda } {f\left( t \right)dt}  - \int_a^b
{f\left( t \right)g\left( t \right)dt} } \right|
\\
\le \left| {\int_a^{a + \lambda } {\left( {\int_a^x {\left( {1 -
g\left( t \right)} \right)dt} } \right)df\left( x \right)} }
\right| +\left| {\int_{a + \lambda }^b {\left( {\int_x^b {g\left(
t \right)dt} } \right)df\left( x \right)} } \right|.
\end{multline*}
Using the fact that for a continuous function $p:[a,b] \to
\mathbb{R}$ and a function $\nu:[a,b] \to \mathbb{R}$ of bounded
variation, one has the inequality:
\begin{align}
\label{eq2.5}\left| {\int_a^b {p\left( t \right)d\nu\left( t
\right)} } \right| \le \mathop {\sup }\limits_{t \in \left[ {a,b}
\right]} \left| {p\left( t \right)} \right| \bigvee_a^b\left( \nu
\right),
\end{align}
observe that $\int_a^x {\left( {1 - g\left( t \right)} \right)dt}$
is a positive increasing function for
$x\in{\left[a,a+\lambda\right]}$ and $\int_x^b {g\left( t
\right)dt}$ is a positive decreasing function for
$x\in{\left[a+\lambda,b\right]}$, it follows that
\begin{align*}
&\left| {\int_a^{a + \lambda } {f\left( t \right)dt}  - \int_a^b
{f\left( t \right)g\left( t \right)dt} } \right|
\\
&\le \mathop {\sup }\limits_{x \in \left[ {a,a + \lambda }
\right]} \left[ {\int_a^x {\left( {1 - g\left( t \right)}
\right)dt} } \right] \cdot \bigvee_a^{a + \lambda } \left( f
\right) + \mathop {\sup }\limits_{x \in \left[ {a + \lambda ,b}
\right]} \left[ {\int_x^b {g\left( t \right)dt} } \right] \cdot
\bigvee_{a + \lambda }^b \left( f \right)
\\
&=\int_a^{a + \lambda } {\left( {1 - g\left( t \right)} \right)dt}
\cdot \bigvee_a^{a + \lambda } \left( f \right) + \int_{a +
\lambda }^b {g\left( t \right)dt} \cdot \bigvee_{a + \lambda }^b
\left( f \right)
\\
&= \left[ {\int_{a + \lambda }^b {g\left( t \right)dt} } \right]
\cdot \bigvee_a^b \left( f \right),
\end{align*}
since $\int_a^{a + \lambda } {\left( {1 - g\left( t \right)}
\right)dt}=\lambda-\int_a^{a + \lambda } {g\left( t \right)dt}=
{\int_{a + \lambda }^b {g\left( t \right)dt} }$ which proves the
first inequality in (\ref{eq2.3}), the second the inequality
follows immediately by assumptions. In similar way and using
(\ref{eq2.2}) we may deduce the desired inequality (\ref{eq2.4}),
and we shall omit the details.
\end{proof}

\begin{corollary}
Let $f,g: [a,b] \to \mathbb{R}$ be such that $0 \le g(t) \le 1$,
for all $t \in [a,b]$ and $\int_a^b {g\left( t \right)df\left( t
\right)}$  exists. If $f$ is decreasing on $[a,b]$, then we have
\begin{align}
\label{eq2.6} 0 \le \int_a^{a + \lambda } {f\left( t \right)dt} -
\int_a^b {f\left( t \right)g\left( t \right)dt}  \le \left[
{f\left( a \right) - f\left( b \right)} \right]\int_{a + \lambda
}^b {g\left( t \right)dt}
\end{align}
and
\begin{align}
\label{eq2.7} 0 \le \int_a^b {f\left( t \right)g\left( t
\right)dt}  - \int_{b - \lambda }^b {f\left( t \right)dt}  \le
\left[ {f\left( a \right) - f\left( b \right)} \right]\int_a^{b -
\lambda } {g\left( t \right)dt}.
\end{align}
\end{corollary}

\subsection{Inequalities for Lipschitzian integrators}

Inequalities for $L$--Lipschitzian integrators may be considered
as follows:
\begin{theorem}
\label{thm.L}Let $f,g : [a,b] \to \mathbb{R}$ be integrable such
that $0 \le g(t) \le 1$ for all $t \in [a,b]$, and $\int_a^b
{g\left( t \right)df\left( t \right)}$ exists. If $f$ is
$L$--Lipschitzian on $[a,b]$, then
\begin{align}
\label{eq2.8}\left| {\int_a^{a + \lambda } {f\left( t \right)dt} -
\int_a^b {f\left( t \right)g\left( t \right)dt} } \right| \le
\frac{L}{2}\left[ {\left( {b - a - \lambda } \right)^2  + \lambda
^2 } \right],
\end{align}
and
\begin{align}
\label{eq2.9}\left| {\int_a^b {f\left( t \right)g\left( t
\right)dt} -  \int_{b - \lambda }^b {f\left( t \right)dt}} \right|
\le \frac{L}{2}\left[ {\left( {b - a - \lambda } \right)^2  +
\lambda ^2 } \right],
\end{align}

where $\lambda : = \int_a^b {g\left( t \right)dt}$.
\end{theorem}

\begin{proof}
Taking the modulus in $(\ref{eq2.1})$ and utilizing the triangle
inequality, we get
\begin{multline*}
\left| {\int_a^{a + \lambda } {f\left( t \right)dt}  - \int_a^b
{f\left( t \right)g\left( t \right)dt} } \right|
\\
\le \left| {\int_a^{a + \lambda } {\left( {\int_a^x {\left( {1 -
g\left( t \right)} \right)dt} } \right)df\left( x \right)} }
\right| + \left| {\int_{a + \lambda }^b {\left( {\int_x^b {g\left(
t \right)dt} } \right)df\left( x \right)} } \right| .
\end{multline*}
Using the fact that for a Riemann integrable function
$p:[c,d]\rightarrow \mathbb{R}$ and $L$-Lipschitzian function $\nu
:[c,d]\rightarrow \mathbb{R}$, one has the inequality
\begin{equation}
\left\vert {\int_{c}^{d}{p\left( t\right) d\nu \left( t\right)
}}\right\vert \leq L\int_{c}^{d}{\left\vert {p\left( t\right)
}\right\vert dt}. \label{eq2.5}
\end{equation}
it follows that
\begin{align*}
&\left| {\int_a^{a + \lambda } {f\left( t \right)dt}  - \int_a^b
{f\left( t \right)g\left( t \right)dt} } \right|
\\
&\le L \left[{\int_a^{a + \lambda } { \left| {\int_a^x {\left( {1
- g\left( t \right)} \right)dt} } \right|dx}+\int_{a + \lambda }^b
{\left| {\int_x^b {g\left( t \right)dt} } \right|dx}} \right]
\\
&= L\left[{\int_a^{a + \lambda } { \left| {\int_a^x {\left( {1 -
g\left( t \right)} \right)dt} } \right|dx}+\int_{a + \lambda }^b
{\left| {\int_x^b {g\left( t \right)dt} } \right|dx} } \right]
\\
&\le L\left[{\int_a^{a + \lambda } {\left( { \int_a^x {\left| {1 -
g\left( t \right)} \right|dt}} \right) dx}+\int_{a + \lambda }^b {
\left( {\int_x^b {\left| {g\left( t \right)} \right|dt}} \right)
dx}} \right].
\end{align*}
But since $0 \le g(x) \le 1$ for all $x\in [a,b]$, (similarly we
have, $ 0 \le 1 - g(x) \le 1$), then
\begin{align*}
&\left| {\int_a^{a + \lambda } {f\left( t \right)dt}  - \int_a^b
{f\left( t \right)g\left( t \right)dt} } \right|
\\
&\le L\left[ {\int_a^{a + \lambda } {\left( { \int_a^x {\left| {1
- g\left( t \right)} \right|dt}} \right) dx}+\int_{a + \lambda }^b
{ \left( {\int_x^b {\left| {g\left( t \right)} \right|dt}} \right)
dx} } \right]
\\
&\le L\left[ {\int_a^{a + \lambda } {\left( { x-a} \right)
dx}+\int_{a + \lambda }^b { \left( {b-x} \right) dx} } \right]
\\
&= \frac{L}{2}\left[ {\lambda ^2 +\left( {b - a - \lambda }
\right)^2} \right],
\end{align*}
which proves (\ref{eq2.8}). In a similar way and using
(\ref{eq2.2}) we may deduce the desired inequality (\ref{eq2.9}),
and we shall omit the details.
\end{proof}

\begin{remark}
Let $f,g$ be as in Theorem \ref{thm.L} . If $\int_a^b {g\left( t
\right)dt} =0 $, then
\begin{align}
\left| { \int_a^b {f\left( t \right)g\left( t \right)dt} } \right|
\le \frac{1}{2}L\left( {b - a} \right)^2.
\end{align}
\end{remark}

\subsection{Inequalities for monotonic non-decreasing integrators}
\begin{theorem}
\label{thm.Mono}Let $f,g : [a,b] \to \mathbb{R}$ be integrable
such that $0 \le g(t) \le 1$ for all $t \in [a,b]$, and $\int_a^b
{g\left( t \right)df\left( t \right)}$ exists. If $f$ is monotonic
nondecreasing on $[a,b]$, then
\begin{multline}
\label{eq2.12}\left| {\int_a^{a + \lambda } {f\left( t \right)dt}
- \int_a^b {f\left( t \right)g\left( t \right)dt} } \right|
\\
\le \lambda \left[ {f\left( {a + \lambda } \right) - f\left( a
\right)} \right] + \left( {b - a - \lambda } \right)\left[
{f\left( b \right) - f\left( {a + \lambda } \right)} \right],
\end{multline}
and
\begin{multline}
\label{eq2.13}\left| {\int_a^b {f\left( t \right)g\left( t
\right)dt} -  \int_{b - \lambda }^b {f\left( t \right)dt}} \right|
\\
\le \lambda\left[ {f\left( b \right) - f\left( {b - \lambda }
\right)} \right] + \left( {b - a - \lambda } \right) \left[
{f\left( {b - \lambda } \right) - f\left( a \right)} \right],
\end{multline}
where $\lambda : = \int_a^b {g\left( t \right)dt}$.
\end{theorem}

\begin{proof}
Taking the modulus in $(\ref{eq2.1})$ and utilizing the triangle
inequality, we get
\begin{multline*}
\left| {\int_a^{a + \lambda } {f\left( t \right)dt}  - \int_a^b
{f\left( t \right)g\left( t \right)dt} } \right|
\\
\le  \left| {\int_a^{a + \lambda } {\left( {\int_a^x {\left( {1 -
g\left( t \right)} \right)dt} } \right)df\left( x \right)} }
\right| + \left| {\int_{a + \lambda }^b {\left( {\int_x^b {g\left(
t \right)dt} } \right)df\left( x \right)} } \right|.
\end{multline*}
Using the fact that for a monotonic non-decreasing function
$\nu:[a,b] \to \mathbb{R}$ and continuous function $p:[a,b] \to
\mathbb{R}$, one has the inequality
\begin{align}
\left| {\int_a^b {p\left( t \right)d\nu\left( t \right)} } \right|
\le \int_a^b {\left| {p\left( t \right)} \right|d\nu \left( t
\right)}. \label{eq2.14}
\end{align}
it follows that
\begin{align*}
&\left| {\int_a^{a + \lambda } {f\left( t \right)dt}  - \int_a^b
{f\left( t \right)g\left( t \right)dt} } \right|
\\
&\le \int_a^{a + \lambda } { \left| {\int_a^x {\left( {1 - g\left(
t \right)} \right)dt} } \right|df\left( x \right)}+ \int_{a +
\lambda }^b {\left| {\int_x^b {g\left( t \right)dt} }
\right|df\left( x \right)}
\\
&= \int_a^{a + \lambda } { \left| {\int_a^x {\left( {1 - g\left( t
\right)} \right)dt} } \right|df\left( x \right)}+\int_{a + \lambda
}^b {\left| {\int_x^b {g\left( t \right)dt} } \right|df\left( x
\right)}
\\
&\le \int_a^{a + \lambda } {\left( { \int_a^x {\left| {1 - g\left(
t \right)} \right|dt}} \right) df\left( x \right)}+\int_{a +
\lambda }^b { \left( {\int_x^b {\left| {g\left( t \right)}
\right|dt}} \right) df\left( x \right)}.
\end{align*}
But since $0 \le g(x) \le 1$ for all $x\in [a,b]$, (similarly we
have, $ 0 \le 1 - g(x) \le 1$), then
\begin{align}
&\left| {\int_a^{a + \lambda } {f\left( t \right)dt}  - \int_a^b
{f\left( t \right)g\left( t \right)dt} } \right|
\nonumber\\
&\le \int_a^{a + \lambda } {\left( { \int_a^x {\left| {1 - g\left(
t \right)} \right|dt}} \right) df\left( x \right)}+\int_{a +
\lambda }^b { \left( {\int_x^b {\left| {g\left( t \right)}
\right|dt}} \right) df\left( x \right)}
\nonumber\\
&\le \int_a^{a + \lambda } {\left( { x-a} \right) df\left( x
\right)}+\int_{a + \lambda }^b { \left( {b-x} \right) df\left( x
\right)} . \label{eq2.15}
\end{align}
Using Riemann--Stieltjes integral we may observe
\begin{align*}
\int_a^{a + \lambda } {\left( { x-a} \right) df\left( x \right)} =
\lambda f\left( {a + \lambda} \right)  - \int_a^{a + \lambda }
{f\left( x \right)dx},
\end{align*}
and
\begin{align*}
\int_{a + \lambda }^b { \left( {b-x} \right) df\left( x \right)} =
- \left( {b-a - \lambda} \right) f\left( a + \lambda \right) +
\int_{a + \lambda }^b { f\left( x \right)dx}.
\end{align*}
Utilizing the monotonicity of $f$ on $[a,b]$, we get
\begin{align*}
\int_a^{a + \lambda } {f\left( x \right)dx} \ge  \lambda f\left( a
\right),\,\,\,\,\,\,\,\,\int_{a + \lambda }^b { f\left( x
\right)dx} \le \left( {b - a - \lambda} \right)f\left( b \right)
\end{align*}
and therefore, by (\ref{eq2.15}) we get
\begin{align*}
&\left| {\int_a^{a + \lambda } {f\left( t \right)dt}  - \int_a^b
{f\left( t \right)g\left( t \right)dt} } \right|
\\
&\le \int_a^{a + \lambda } {\left( { x-a} \right) df\left( x
\right)}+\int_{a + \lambda }^b { \left( {b-x} \right) df\left( x
\right)}
\\
&\le  \lambda \left[ {f\left( {a + \lambda } \right) - f\left( a
\right)} \right] + \left( {b - a - \lambda } \right)\left[
{f\left( b \right) - f\left( {a + \lambda } \right)} \right],
\end{align*}
which proves (\ref{eq2.12}). In a similar way and using
(\ref{eq2.2}) we may deduce the desired inequality (\ref{eq2.13}),
and we shall omit the details.
\end{proof}

\begin{remark}
Let $f,g$ be as in Theorem \ref{thm.Mono} . If $\int_a^b {g\left(
t \right)dt} =0 $, then
\begin{align}
\left| { \int_a^b {f\left( t \right)g\left( t \right)dt} } \right|
\le \left( {b - a} \right)\left[ {f\left( b \right) - f\left( {a }
\right)} \right].
\end{align}
\end{remark}

\subsection{Inequalities for absolutely continuous integrators}

Another result for absolutely continuous integrators is
incorporated in the following theorem:
\begin{theorem}
\label{thm5}Let $f,g: [a,b] \to \mathbb{R}$ be integrable such
that $0 \le g(t) \le 1$, for all $t \in [a,b]$ such that $\int_a^b
{g\left( t \right)df\left( t \right)}$ exists. If $f$ is
absolutely continuous on $[a,b]$ with $f' \in L_p[a,b]$, $1 \le p
\le \infty$, then we have
\begin{multline}
\label{eq2.16}\left| {\int_a^{a + \lambda } {f\left( t \right)dt -
\int_a^b {f\left( t \right)g\left( t \right)dt} } } \right|
\\
\le \left\{
\begin{array}{l}
 \frac{1}{2}\left[ {\lambda ^2  + \left( {b - a - \lambda }
\right)^2 } \right] \left\| {f'} \right\|_{\infty ,\left[ {a,b}
\right]},\,\,\,\,\,\,\,\,\,\,\,\,\,\,\,\,\,\,\,\,\,\,\,\,\,\,\,\,\,if\,\,\,\,\,\,\,\,f' \in L_\infty  [a,b]; \\
  \\
 \frac{\left\| {f'} \right\|_{p ,\left[ {a,b}
\right]}}{{\left( {q + 1} \right)^{1/q} }} \left[ {\lambda
^{\left( {q + 1} \right)/q}  + \left( {b - a - \lambda }
\right)^{\left( {q + 1} \right)/q} } \right],\,\,\,\,\,\,\,\,if\,\,\,\,\,\,\,\,f' \in L_p [a,b],p > 1; \\
  \\
 \left[ \int_{a + \lambda }^b {g\left( t \right)dt}
\right]\left\| {f'} \right\|_1 ,\,\,\,\,\,\,\,\,\,\,\,\,\,\,\,\,\,\,\,\,\,\,\,\,\,\,\,\,\,\,\,\,\,\,\,\,\,\,\,\,\,\,\,\,\,\,\,\,\,\,\,\,\,\,\,\,\,\,\,\,\,\,\,\,\,\,\,\,\,if\,\,\,\,\,\,\,f' \in L_1 [a,b] \\
 \end{array} \right.
\end{multline}
and
\begin{multline}
\label{eq2.17}\left| {\int_a^b {f\left( t \right)g\left( t
\right)dt - \int_{b - \lambda }^b {f\left( t \right)dt} } }
\right|
\\
\le \left\{
\begin{array}{l}
 \frac{1}{2}\left[ {\lambda ^2  + \left( {b - a - \lambda }
\right)^2 } \right] \left\| {f'} \right\|_{\infty ,\left[ {a,b}
\right]},\,\,\,\,\,\,\,\,\,\,\,\,\,\,\,\,\,\,\,\,\,\,\,\,\,\,\,\,\,if\,\,\,\,\,\,\,\,f' \in L_\infty  [a,b]; \\
  \\
 \frac{\left\| {f'} \right\|_{p ,\left[ {a,b}
\right]}}{{\left( {q + 1} \right)^{1/q} }} \left[ {\lambda
^{\left( {q + 1} \right)/q}  + \left( {b - a - \lambda }
\right)^{\left( {q + 1} \right)/q} } \right],\,\,\,\,\,\,\,\,if\,\,\,\,\,\,\,\,f' \in L_p [a,b],p > 1; \\
  \\
 \left[ \int_a^{b - \lambda } {g\left( t \right)dt}
\right]\left\| {f'} \right\|_1 ,\,\,\,\,\,\,\,\,\,\,\,\,\,\,\,\,\,\,\,\,\,\,\,\,\,\,\,\,\,\,\,\,\,\,\,\,\,\,\,\,\,\,\,\,\,\,\,\,\,\,\,\,\,\,\,\,\,\,\,\,\,\,\,\,\,\,\,\,\,if\,\,\,\,\,\,\,f' \in L_1 [a,b] \\
 \end{array} \right.
\end{multline}

\end{theorem}

\begin{proof}
Assume that $f' \in L_{\infty}[a,b]$ and utilizing the triangle
inequality, we get
\begin{align*}
&\left| {\int_a^{a + \lambda } {f\left( t \right)dt}  - \int_a^b
{f\left( t \right)g\left( t \right)dt} } \right|
\\
&\le \left| {\int_a^{a + \lambda } {\left( {\int_a^x {\left( {1 -
g\left( t \right)} \right)dt} } \right)f'\left( x \right)dx} }
\right|+ \left| {\int_{a + \lambda }^b {\left( {\int_x^b {g\left(
t \right)dt} } \right)f'\left( x \right)dx} } \right|
\\
&\le \int_a^{a + \lambda } { \left| {\int_a^x {\left( {1 - g\left(
t \right)} \right)dt} } \right|\left| {f'\left( x \right)}
\right|dx} + \int_{a + \lambda }^b {\left| {\int_x^b {g\left( t
\right)dt} } \right| \left| {f'\left( x \right)} \right| dx}
\\
&\le  \left\| {f'} \right\|_{\infty ,\left[ {a,a + \lambda }
\right]}\int_a^{a + \lambda } { \left( {\int_a^x {\left| {1 -
g\left( t \right)} \right|dt}} \right)dx} + \left\| {f'}
\right\|_{\infty ,\left[ {a + \lambda ,b} \right]} \int_{a +
\lambda }^b {\left( {\int_x^b {\left| {g\left( t \right)}
\right|dt} } \right)dx}
\\
&\le  \left\| {f'} \right\|_{\infty ,\left[ {a,a + \lambda }
\right]}\int_a^{a + \lambda } { \left( {x - a} \right)dx} +
\left\| {f'} \right\|_{\infty ,\left[ {a + \lambda ,b} \right]}
\int_{a + \lambda }^b {\left( {b - x } \right)dx}
\\
&\le  \frac{1}{2}\left[ {\lambda ^2  + \left( {b - a - \lambda }
\right)^2 } \right] \left\| {f'} \right\|_{\infty ,\left[ {a,b}
\right]},
\end{align*}
and the first inequality in  (\ref{eq2.16}) is proved.

Assume $f' \in L_{p}[a,b]$, using H\"{o}lder integral inequality
for $p>1$, $\frac{1}{p} + \frac{1}{q} = 1$, we also have
\begin{align*}
&\left| {\int_a^{a + \lambda } {f\left( t \right)dt}  - \int_a^b
{f\left( t \right)g\left( t \right)dt} } \right|
\\
&\le \int_a^{a + \lambda } { \left| {\int_a^x {\left( {1 - g\left(
t \right)} \right)dt} } \right|\left| {f'\left( x \right)}
\right|dx} + \int_{a + \lambda }^b {\left| {\int_x^b {g\left( t
\right)dt} } \right| \left| {f'\left( x \right)} \right| dx}
\\
&\le  \left( {\int_a^{a + \lambda } { \left| {\int_a^x {\left( {1
- g\left( t \right)} \right)dt} } \right|^q dx}} \right)^{1/q}
\left( {\int_a^{a + \lambda } {\left| {f'\left( x \right)}
\right|^p dx}} \right)^{1/p}
\\
&\qquad+ \left( {\int_{a + \lambda }^b {\left| {\int_x^b {g\left(
t \right)dt} } \right|^q dx}} \right)^{1/q} \left( {\int_{a +
\lambda }^b { \left| {f'\left( x \right)} \right|^p dx}}
\right)^{1/p}
\\
&\le  \left\| {f'} \right\|_{p ,\left[ {a,a + \lambda } \right]}
\left( {\int_a^{a + \lambda } { \left( {x - a} \right)^qdx} }
\right)^{1/q} + \left\| {f'} \right\|_{p,\left[ {a + \lambda ,b}
\right]} \left( {\int_{a + \lambda }^b {\left( {b - x }
\right)^qdx}} \right)^{1/q}
\\
&\le \frac{\left\| {f'} \right\|_{p ,\left[ {a,b}
\right]}}{{\left( {q + 1} \right)^{1/q} }} \left[ {\lambda
^{\left( {q + 1} \right)/q}  + \left( {b - a - \lambda }
\right)^{\left( {q + 1} \right)/q} } \right],
\end{align*}
giving the second inequality in (\ref{eq2.16}).

Finally, we also observe that
\begin{align*}
&\left| {\int_a^{a + \lambda } {f\left( t \right)dt}  - \int_a^b
{f\left( t \right)g\left( t \right)dt} } \right|
\\
&\le \mathop {\sup }\limits_{x \in \left[ {a,a + \lambda }
\right]} \left[ {\int_a^x {\left( {1 - g\left( t \right)}
\right)dt} } \right] \cdot \int_a^{a + \lambda } \left|
f'\left(x\right) \right|dx + \mathop {\sup }\limits_{x \in \left[
{a + \lambda ,b} \right]} \left[ {\int_x^b {g\left( t \right)dt} }
\right] \cdot \int_{a + \lambda }^b
\left|f'\left(x\right)\right|dx
\\
&=\int_a^{a + \lambda } {\left( {1 - g\left( t \right)} \right)dt}
\cdot \int_a^{a + \lambda } \left| f'\left(x\right) \right|dx+
\int_{a + \lambda }^b {g\left( t \right)dt} \cdot \int_{a +
\lambda }^b \left|f'\left(x\right)\right|dx
\\
&= \left[ \int_{a + \lambda }^b {g\left( t \right)dt}
\right]\left\| {f'} \right\|_1.
\end{align*}
which proves the last inequality in (\ref{eq2.16}). The
inequalities in (\ref{eq2.17}) may be proved in the same way using
the identity (\ref{eq2.2}), we shall omit the details.
\end{proof}

\begin{remark}
One may deduce new inequalities of Hayashi's type by using the
substituting $ {{g\left( t \right)} \mathord{\left/
 {\vphantom {{g\left( t \right)} A}} \right.
 \kern-\nulldelimiterspace} A}$ for $g\left( t \right)$, where $A$ is nonzero positive
 constant, and then using the identities
 \begin{multline}
\label{eq2.18}\int_a^{a + \lambda } {f\left( t \right)dt}  -
\frac{1}{A} \int_a^b {f\left( t \right)g\left( t \right)dt}
\\
= - \frac{1}{A} \int_a^{a + \lambda } {\left( {\int_a^x {\left( {A
- g\left( t \right)} \right)dt} } \right)df\left( x \right)} -
\frac{1}{A} \int_{a + \lambda }^b {\left( {\int_x^b {g\left( t
\right)dt} } \right)df\left( x \right)} ,
\end{multline}
and
\begin{multline}
\label{eq2.19} \frac{1}{A}\int_a^b {f\left( t \right)g\left( t
\right)dt - \int_{b - \lambda }^b {f\left( t \right)dt} }
\\
=  - \frac{1}{A}\int_a^{b - \lambda } {\left( {\int_a^x {g(t)dt} }
\right)df\left( x \right)}  - \frac{1}{A} \int_{b - \lambda }^b
{\left( {\int_x^b {\left( {A - g\left( x \right)} \right)dt} }
\right)df\left( x \right)},
\end{multline}
where $\lambda : = \frac{1}{A}\int_a^b {g\left( t \right)dt}$.

\end{remark}

\end{document}